\documentclass{amsart}
\usepackage{graphicx}
\usepackage{pslatex}
\usepackage{amsmath,amsfonts}
\usepackage{rotating}
\usepackage{amssymb}
\usepackage{verbatim}
\usepackage{rotating}
\usepackage{color,xcolor}
\usepackage{mathrsfs}
\usepackage{hyperref}
\usepackage{tikz}
\tikzset{node distance=2cm, auto}
\usetikzlibrary{matrix}

\newtheorem{theorem}{Theorem}

\newtheorem{proposition}[theorem]{Proposition}

\theoremstyle{definition}
\newtheorem{definition}[theorem]{Definition}
\newtheorem{example}[theorem]{Example}

\theoremstyle{definition}
\newtheorem{visual}[theorem]{Visualization}
\theoremstyle{remark}
\newtheorem{remark}[theorem]{Remark}

\numberwithin{theorem}{section}

\newcommand{\X}{\textbf{{X}}}

\def\({\left(}
\def\){\right)}
\def\[{\left[}
\def\]{\right]}
\def\<{\left<}
\def\>{\right>}

\def\st#1#2{\substack{#1\\#2}}

\begin{document}

\title{Coefficients for Higher Order Hochschild Cohomology}

\author{Bruce R. Corrigan-Salter}

\email{ brcs@wayne.edu}

\address{Department of Mathematics, Wayne State University, 656 W. Kirby, Detroit, MI 48202, USA}

\subjclass[2010]{Primary: 13D03; Secondary: 18G30, 55U10, 13D10, 16S80}

\keywords{Hochschild, cohomology, higher order, simplicial, deformation, multi-module, coefficient}

\begin{abstract}
When studying deformations of an $A$-module $M$, Laudal and Yau showed that one can consider $1$-cocycles in the Hochschild cohomology of $A$ with coefficients in the bi-module $End_k(M).$  With this in mind, the use of higher order Hochschild (co)homology, presented by Pirashvili and Anderson, to study deformations seems only natural though the current definition allows only symmetric bi-module coefficients.  In this paper we present an extended definition for higher order Hochschild cohomology which allows multi-module coefficients (when the simplicial sets $\X_{\bullet}$ are accommodating) which agrees with the current definition.  Furthermore we determine the types of modules that can be used as coefficients for the Hochschild cochain complexes based on the simplicial sets they are associated to.

\end{abstract}

\maketitle

\section{Introduction}
\label{INTRO}

In \cite{Pirashvili} Pirashvili makes explicit a definition of higher order Hochschild homology of a $k$-algebra $A$ with coefficients in an $A$-module $M$, implicitly defined in \cite{Anderson} by Anderson.  This is done by considering the composition of functors:
$$\mathcal{L}(A,M)\colon \Gamma \rightarrow Vect$$
$$\{0,1,\cdots,n\}\rightarrow M\otimes A^{\otimes n}$$ and
$$Y\colon \Delta^{op}\rightarrow \Gamma$$
where we let $\Gamma$ denote the category of finite sets and where $Y$ is a pointed simplicial set.  This composition yields a simplicial vector space, where the homology of the associated chain complex serves as the definition of higher order Hochschild homology.  

One can naturally extend the definition given in \cite{Pirashvili} to also define a notion of higher order Hochschild cohomology (see \cite{Ginot} for a precise definition).  A reason for doing so would be to consider deformation theory in this new setting.  This connection between traditional Hochschild cohomology and deformation theory has been studied for some time.  In \cite{Gerstenhaber}, Gerstenhaber illustrates the connection between deformations of an associative $k$-algebra $A$ and $2$-cocycles in the Hochschild cohomology of $A$ with trivial coefficients by showing that $HH^2(A,A)$ is the group of isomorphism classes of square-zero deformations of $A$, and in \cite{Yau} and \cite{Laudal}, Yau and Laudal, respectively, show the connection between deformations of an $A$-module $M$ with $1$-cocycles of the Hochschild cochain complex of $A$ with coefficients in $End_k(M).$  More precisely, first order deformations of an $A$-module $M$ are in bijection with $HH^1(A,End_k(M)).$

Upon examination, it can be seen that the current definition for higher order Hochschild cohomology allows only symmetric bi-module coefficients, however to study deformations of a module we need to consider the nonsymmetric bi-modules $End_k(M).$

One knows that classical Hochschild cohomology takes coefficients in a bi-modules, so we ask the following question.  What coefficients can higher order Hochschild cohomology take?  The goal of this paper thus becomes two fold.  We aim to extend the definition of higher order Hochschild cohomology to include multi-modules (not necessarily symmetric) as coefficients and in particular to determine what type of coefficient modules can be used when choosing a simplicial set $\X_{\bullet}$ to construct a cosimplicial $k$-vector space over.
With this in mind, we aim to show the following Theorem.

\begin{theorem}
\label{MAINTHM}
Let $A$ be a commutative $k$-algebra.  Given a pointed simplicial set $\X_{\bullet}$, there exists a cosimplicial $k$-vector space $(M,X)^{\bullet}$ associated to an $A$-module $M$ given by 

$$(M,X)^n= \hom_k (k \otimes_k 
\bigotimes\limits_{\st{\sigma \in \X_n}{\sigma \neq \ast}} A,M)$$ 

with coface and codegeneracy maps given by

$$d_n^i f(1\otimes_k \bigotimes\limits_{\st{\sigma \in \X_{n+1}}{\sigma \neq \ast}}a_{\sigma})=\prod\limits_{\st{\sigma \in \X_{n+1}}{d_i (\sigma)=\ast}} (\Lambda_{(i,n)}^{\sigma}(a_{\sigma}))\cdot f(1\otimes_k \bigotimes\limits_{\st{\Omega\in\X_n}{\Omega \neq \ast}} \prod\limits_{\st{\sigma \in \X_{n+1}}{d_i (\sigma)=\Omega}}a_{\sigma})$$

and

$$s_n^i f(1\otimes_k \bigotimes\limits_{\st{\sigma \in \X_{n+1}}{\sigma\neq\ast}}a_{\sigma})=f(1\otimes_k\bigotimes\limits_{\st{\Omega \in \X_{n+1}}{\Omega\neq\ast}} 1\cdot \prod\limits_{\st{\sigma \in \X_n}{s_i(\sigma)=\Omega}}a_{\sigma})$$

if the actions $\Lambda_{(-,-)}^-$ on $M$ satisfy the following for simplices $\sigma, \Omega$ and $\mu$:

\begin{itemize}
\item[i)]  $\Lambda_{(j,n+1)}^{\sigma} = \Lambda_{(i,n+1)}^{\sigma}$ if $\sigma \neq \ast , d_i (\sigma)=d_j(\sigma)=\ast$ and the dimension of $\sigma$ is at least 2 and $i<j.$  We call this a \textit{sweep around}.
\item[ii)]  $\Lambda_{(j,n+1)}^{\sigma}=\Lambda_{(j-1),n}^{\Omega}$ if $d_i(\sigma)=\Omega , d_j(\sigma)=\ast , d_{j-1}(\Omega)=\ast$ and the dimension of $\sigma$ is at least 2.  We call this a \textit{sweep out 1}.
\item[iii)]  $\Lambda_{(i,n)}^{\Omega}=\Lambda_{(j-1,n)}^{\mu}$ if $d_i(\Omega)=\ast , d_{j-1}(\mu)=\ast$ and there exists a $\sigma$ of dimension at least 2 where $d_j(\sigma)=\Omega$, $d_i(\sigma)=\mu$ and $i<j.$  We call this a \textit{sweep across} .
\item[iv)]  $\Lambda_{(i,n)}^{\Omega}=\Lambda_{(i,n+1)}^{\sigma}$ if $d_i(\sigma)=\ast , d_i(\Omega)=\ast , d_j(\sigma)=\Omega$ and the dimension of $\sigma$ is at least 2.  We call this a \textit{sweep out 2}.
\end{itemize}
where $\Lambda_{(i,n)}^\sigma (a))$ represents the $(\Lambda_{(i,n)}^\sigma$ action of $a$ whenever $0\leq i\leq 1$ and $\sigma \in X_{n+1}$ and $d_i(\sigma)=\ast$.  We take a product of such actions to represent the composition of the actions, which we assume to be commutative (i.e. if $M$ has two actions, we actually assume that $M$ is a bimodule).

\end{theorem}

\begin{remark}
For $$(M,X)^n= \hom_k (k \otimes_k 
\bigotimes\limits_{\st{\sigma \in \X_n}{\sigma \neq \ast}} A,M)$$ in Theorem \ref{MAINTHM} we assume that the trivial $k$ as a tensor factor in the domain represents the base point.
\end{remark}

While the list of axioms in Theorem \ref{MAINTHM} may seem random, we give a intuitive visual description of each axiom in Section \ref{VIS} and in Section \ref{EXAMPLES} we use the visual descriptions to give results for a variety of simplicial sets.
  
We can now define the following. 

\begin{definition}
The cohomology of the cochain complex associated to $(M,X)^{\bullet}$ (by taking alternating sums of coface maps) is the {\em higher order Hochschild cohomology} of $A$ with coefficients in $M,$ which we denote as $HH^{\ast}_{X}(A,M).$
\end{definition}

For the remainder of this paper, we fix the field $k$ and assume $A$ is a commutative algebra over $k.$  We also assume the $A$-module $M$ can have multiple actions as our quest is to determine what actions may be present in the associated higher order Hochschild cochain complex.  

\section{Comparison to Classical Definitions and Applications}

\subsection{Comparison}
We would like to see that the definition given in section \ref{INTRO} agrees with both the classical Hochschild cohomology definition as well as the higher order Hochschild cohomology definition given by Pirashvili.  We provide the connections through the following examples.

\begin{example}
Let $X_{\bullet}=S^1$ be the simplicial set of $S^1$ which contains one 0-simplex and one non-degenerate 1-simplex.  Given an $A$ bi-module $M,$ we see that $HH_{X}^{\ast}(A,M)$ is classical Hochschild cohomology.  The fact that we are able to work over bi-modules, which are not necessarily symmetric is illustrated by Example \ref{classical}.
\end{example}
  
\begin{example}
For a simplicial set $X_{\bullet}$ and symmetric $A$ bi-module $M$ , we see that $HH^{\ast}_{X}(A,M)$ agrees with the definition of higher order Hochschild cohomology given implicitly by Pirashvili in \cite{Pirashvili}.
\end{example}

\subsection{Applications}
As stated in section \ref{INTRO}, one hope of defining higher order Hochschild cohomology with multi-module coefficients is to discover additional connections to deformation theory.  To begin, consider the following definition.

\begin{definition}
\label{alg}
Given a $k$-algebra $A$ and $A$-module $I,$ we define a $\emph{first order deformation}$ of $A$ by $I$ to be a short exact sequence $$0\rightarrow I \rightarrow A' \xrightarrow{\psi} A \rightarrow 0$$ which splits as $k$-vector spaces and with the properties that $\psi$ is a ring map and $I^2=0$ as an ideal of $A'.$
\end{definition}

We see that defining a deformation of $A$ by $I$ amounts to defining a multiplication for $A\oplus I,$ given by $(a_0,i_0)(a_1,i_1)=(a_0a_1,a_oi_1+i_0a_1+f(a_0\otimes a_1)),$ but such a multiplication gives a map $f\colon A \otimes A \rightarrow I$ so that $f \in HH^2(A,I).$
Now, if we consider the trivial first order deformation of $A$ by $A$ i.e. $f=0\in HH^2(A,A),$ we get a short exact sequence $$0\rightarrow A \rightarrow A[x]/x^2 \xrightarrow{\psi} A \rightarrow 0,$$ which induces a map of modules $$mod(A[x]/x^2) \rightarrow mod(A)$$ which takes an $A[x]/x^2$ module $M$ to $M/xM.$  This brings us to the following definition.

\begin{definition}
Given an $A$-module $M,$ a first order deformation of $M$ is an $A[x]/x^2$-module $M'$ so that $M'/xM' \cong M.$
\end{definition}
As stated in section \ref{INTRO}, we get the following.

\begin{proposition}\cite[3.1]{Yau}
Isomorphism classes of first order deformations of an $A$-module $M$ are in bijection with $HH^1(A, End_k(M)).$
\end{proposition}

Seeing the connections between classical Hochschild cohomology and deformation theory, one can then ask the question. Are there similar connections to higher order Hochschild cohomology when considering additional deformations?  Certainly we could consider modifying Definition \ref{alg} to let $A' \cong A\oplus I$ where $I^3=0$ instead of $I^2=0$ (cubed zero deformations) or let $A=A[x,y]/x^2,y^2$ (or any Artinian algebra).  We could even consider deformations using Steenrod relations.  This would, in principle give a means of computing all isomorphism classes of finitely generated modules over an Artinian algebra.  Current work of the author and Salch seek to answer such questions and discover what structures exist on $HH_{X}^{\ast}(A,M).$
\section{Proof of Theorem \ref{MAINTHM}}
\label{PROOF}
In order for the cosimplicial structure to exist, we simply need the cosimplicial identities to be satisfied, precisely:

\begin{itemize}
\item[a)] $d^jd^i=d^id^{j-1}$ for $i<j$
\item[b)] $s^js^i=s^{i-1}s^j$ for $i\geq j$
\item[c)] $s^jd^i=\left\{\begin{matrix}d^is^{j-1}  & $for$ & i<j & & \\
$id$ & $for$ & i=j & $or$ & i=j+1 \\ d^{i-1}s^j & $for$ & i>j+1 & & \end{matrix}\right.$
\end{itemize}

With this in mind, we prove Theorem \ref{MAINTHM}:

\begin{proof}[Proof of Theorem \ref{MAINTHM}]

When composing the coface and codegeneracy maps from Section \ref{INTRO} it becomes clear what action identifications must be made in order for the cosimplicial structure to exist.

For b) notice when $i\geq j$

$$s_{n-1}^js_n^i f(1\otimes_k \bigotimes\limits_{\st{\sigma\in \X_n}{\sigma\neq\ast}} a_{\sigma})=f(1\otimes_k \bigotimes\limits_{\st{\mu \in \X_{n+1}}{\mu \neq \ast}}1\cdot \prod\limits_{\st{\Omega \in \X_n}{s_i(\Omega)=\mu}}\prod\limits_{\st{\sigma \in \X_{n-1}}{s_j(\sigma)=\Omega}}a_{\sigma})$$

but this is precisely 

$$f(1\otimes_k \bigotimes_{\st{\mu \in \X_{n+1}}{\mu\neq\ast}}1\cdot \prod\limits_{\st{\sigma \in \X_{n-1}} {s_is_j(\sigma)=\mu}}a_{\sigma})$$
similarly we get 
$$s_{n-1}^{i-1}s_n^j f(1\otimes_k \bigotimes\limits_{\st{\sigma\in \X_n}{\sigma\neq\ast}} a_{\sigma})=f(1\otimes_k \bigotimes_{\st{\mu \in \X_{n+1}}{\mu\neq\ast}}1\cdot \prod\limits_{\st{\sigma \in \X_{n-1}}{s_js_{i-1}(\sigma)=\mu}}a_{\sigma}).$$
Since $s_is_j(\sigma)=s_js_{i-1}(\sigma)$ for all $\sigma \in \X_{n-1}$ we get $s^js^i=s^{i-1}s^j.$

Now for part a) notice when $i<j$

$$d_{n+1}^jd_n^if(1\otimes_k \bigotimes\limits_{\st{\sigma \in \X_{n+2}}{\sigma\neq\ast}}a_{\sigma})$$ $$=\prod\limits_{\st{\sigma \in \X_{n+2}}{d_j(\sigma)=\ast}}(\Lambda_{(j,n+1)}^{\sigma}(a_{\sigma}))\prod\limits_{\st{\Omega \in \X_{n+1}} {d_i(\Omega)=\ast}} (\Lambda_{(i,n)}^{\Omega}(\prod\limits_{\st{\sigma \in \X_{n+2}}{d_j(\sigma)=\Omega}}a_{\sigma}))\cdot f(1\otimes_k \bigotimes\limits_{\st{\mu \in \X_n}{\mu\neq\ast}}\prod\limits_{\st{\sigma \in \X_{n+2}} {d_id_j(\sigma)=\mu}}a_{\sigma})$$
and
$$d_{n+1}^id_n^{j-1}f(1\otimes_k \bigotimes\limits_{\st{\sigma \in \X_{n+2}}{\sigma\neq\ast}}a_{\sigma})$$ $$=\prod\limits_{\st{\sigma \in \X_{n+2}}{d_i(\sigma)=\ast}}(\Lambda_{(i,n+1)}^{\sigma}(a_{\sigma}))\prod\limits_{\st{\Omega \in \X_{n+1}}{d_{j-1}(\Omega)=\ast}} (\Lambda_{(j-1,n)}^{\Omega}(\prod\limits_{\st{\sigma \in \X_{n+2}} {d_i(\sigma)=\Omega }}a_{\sigma}))\cdot f(1\otimes_k \bigotimes\limits_{\st{\mu \in \X_n}{\mu\neq\ast}}\prod\limits_{\st{\sigma \in \X_{n+2}}{d_{j-1}d_i(\sigma)=\mu}}a_{\sigma})$$
Now, if there exists a $\sigma \in \X_{n+2}$ so that $d_i(\sigma)=d_j(\sigma)=\ast$ then we must have that 
$$d_{n+1}^jd_n^if(1\otimes_k \otimes a\bigotimes\limits_{\substack{\gamma \in \X_{n+2}\\ \gamma \neq \ast \\ \gamma \neq \sigma}}1)=d_{n+1}^id_n^{j-1}f(1\otimes_k \otimes a\bigotimes\limits_{\substack{\gamma \in \X_{n+2} \\ \gamma \neq \ast \\ \gamma \neq \sigma}}1)$$
(where $a$ is the element for the tensor factor associated to $\sigma$) but this gives us that 
$$(\Lambda_{(j,n+1)}^{\sigma}(a))f(1)=(\Lambda_{(i,n+1)}^{\sigma}(a))f(1)$$
which implies $$\Lambda_{(j,n+1)}^{\sigma}=\Lambda_{(i,n+1)}^{\sigma}$$
so there is a set of actions that must be identified in order for a cosimplicial structure to exist.  The identification shown here is actually part i) of Theorem \ref{MAINTHM}.  Following an analogous argument it can be seen that parts ii), iii) and iv) are also consequences that come from ensuring $d^jd^i=d^id^{j-1}$ for $i<j.$  With that, part a) of the cosimplicial identities is satisfied as long as these actions are identified since $d_id_j(\sigma)=d_{j-1}d_i(\sigma)$ for all $\sigma \in \X_{n+2}$

Lastly, for part c) we see that 
$$s_n^jd_n^if(1\otimes_k \bigotimes\limits_{\st{\sigma \in \X_{n}}{\sigma \neq \ast}}a_{\sigma})$$
$$= \prod\limits_{\st{\Omega \in \X_{n+1}} {d_i(\sigma)=\ast}} (\Lambda_{(i,n)}^{\Omega}(1\cdot \prod\limits_{\st{\sigma \in \X_n}{s_j(\sigma)=\Omega}}a_{\sigma}))\cdot f(1\otimes_k \bigotimes\limits_{\st{\mu \in \X_n}{\mu\neq\ast}}\prod\limits_{\st{\sigma \in \X_n}{d_is_j(\sigma)=\mu}}a_{\sigma})$$
and:
$$d_{n-1}^is_{n-1}^{j-1}f(1\otimes_k \bigotimes\limits_{\st{\sigma \in \X_n}{\sigma\neq\ast}}a_{\sigma})$$
$$=\prod\limits_{\st{\sigma \in \X_n}{d_i(\sigma)=\ast}}(\Lambda_{(i,n-1)}^{\sigma}(a_{\sigma}))\cdot f(1\otimes_k \bigotimes\limits_{\st{\mu \in \X_n}{\sigma\neq\ast}}1\cdot \prod\limits_{\st{\sigma \in \X_n}{s_{j-1}d_i(\sigma)=\mu}}a_{\sigma})$$
finally
$$d_{n-1}^{i-1}s_{n-1}^{j}f(1\otimes_k \bigotimes\limits_{\st{\sigma \in \X_n}{\sigma\neq\ast}}a_{\sigma})$$
$$=\prod\limits_{\st{\sigma \in \X_n}{d_{i-1}(\sigma)=\ast}}(\Lambda_{(i-1,n-1)}^{\sigma}(a_{\sigma}))\cdot f(1\otimes_k \bigotimes\limits_{\st{\mu \in \X_n}{\mu \neq \ast}}1\cdot \prod\limits_{\st{\sigma \in \X_n}{s_{j}d_{i-1}(\sigma)=\mu}}a_{\sigma})$$
Using the formulas above, we get the following identities on actions:
\begin{itemize}
\item[v)] $\Lambda_{(i,n)}^{\Omega}=\Lambda_{(i,n-1)}^{\sigma}$ for $i,j$, $d_i(\sigma)=\ast$, $s_j(\sigma)=\Omega$, $d_i(\Omega)=\ast$ and the dimension of $\Omega$ is at least 1.
\item[vi)] $\Lambda_{(i,n)}^{\Omega}=\Lambda_{(i-1,n-1)}$ for $i>j+1$, $d_{i-1}(\sigma)=\ast$, $s_j(\sigma)=\Omega$, $d_i(\Omega)=\ast$ and the dimension of $\Omega$ is at least 1.
\end{itemize}
but v) and vi) are consequences of ii) and iv) from the list in Theorem \ref{MAINTHM} since $\Omega$ must be of dimension larger than 1 in order to be a degenerate simplex and have $d_i(\Omega)=\ast$.   
\end{proof}

\section{Visualization of Action Identifications}
\label{VIS}

\begin{remark}
Notice that a $\Lambda_{(i,n-1)}^\sigma$ action exists any time there is an $n$-simplex $\sigma$ with the property that $\sigma \neq \ast$ and $d_i(\sigma)=\ast$ for some $i.$  We will refer to this as the $i^{th}$ action of $\sigma$ and visually we can think of this action as being "pointed" towards the $i^{th}$ face of $\sigma.$
\end{remark}

We now consider ways to visualize the identifications of the actions, presented in  Section \ref{PROOF}.  

\begin{visual}
For iii) (sweep across) if two faces $\mu$ and $\Omega$ of an $n$-simplex $\sigma$ have a common face of $\ast ,$ then the action of $\mu$ which points towards $\ast$ is the same as the action of $\Omega$ which points towards $\ast.$  This is illustrated below with $\sigma$, a 2-simplex, $\mu ,$ the $0^{th}$ face and $\Omega ,$ the $2^{nd}$ face, while the 0-simplex labeled $1$ represents $\ast .$  Notice the $0^{th}$ face of $\Omega$ is $\ast$ and the $1^{st}$ face of $\mu$ is $\ast .$  This gives us that the associated actions, which point towards $\ast$ are identified.  We can see that in this particular instance, the action of $\Omega$  points in the direction of the orientation of $\Omega$ and the action of $\mu$ points against the orientation of $\mu .$  For this reason, when dealing with simplices of dimension $1$ we will refer to the action as either being forward or backward.  In this case the forward action of $\Omega $ is identified with the backwards action of $\mu .$

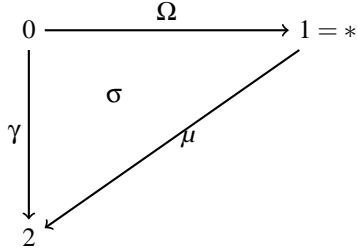
\begin{equation*}
\begin{tikzpicture}[baseline=(current bounding box.center)]
\matrix (m) 
[matrix of math nodes, row sep=1em, column sep=1em, text height=1.5ex, text depth=0.25ex]
{
0 & &        & & & & & & 1=\ast \\
  & & \sigma & & & & & &  \\
  & &        & & & & & &   \\
  & &        & & & & & &  \\ 
  & &        & & & & & &  \\
2 & &        & & & & & & \\
};
\path[->, thick,]
(m-1-1)
edge node[anchor=south] {$\Omega$}  (m-1-9)
edge node[anchor=east] {$\gamma$} (m-6-1)
(m-1-9)
edge node[anchor=west] {$\mu$} (m-6-1)
; 
\end{tikzpicture}
\end{equation*}
\end{visual}

\begin{visual}
For ii) and iv) (sweep out 1 and sweep out 2) we see that if an $n$-simplex $\sigma$ has an action which points towards a face $\gamma ,$ then any other face of $\sigma$, which is not equal to $\ast$ has an equal action pointing in the corresponding direction, towards some face of $\gamma .$  This is illustrated below with $\sigma ,$ a $2$-simplex whose first face is $\gamma$ which is in fact $\ast.$ We also have that $\mu$ is the $0^{th}$ face of $\sigma$  and $\Omega$ is the $2^{nd}$ face of $\sigma .$  Notice that the action of $\sigma ,$ which points towards $\gamma$ is the same as the backward action of $\Omega$ and the forward action of $\mu .$

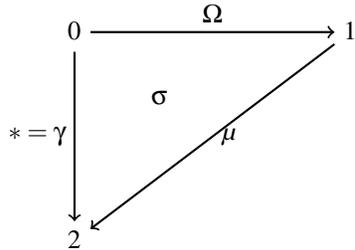
\begin{equation*}
\begin{tikzpicture}[baseline=(current bounding box.center)]
\matrix (m) 
[matrix of math nodes, row sep=1em, column sep=1em, text height=1.5ex, text depth=0.25ex]
{
0 & &        & & & & & & 1 \\
  & & \sigma & & & & & &  \\
  & &        & & & & & &   \\
  & &        & & & & & &  \\ 
  & &        & & & & & &  \\
2 & &        & & & & & & \\
};
\path[->, thick,]
(m-1-1)
edge node[anchor=south] {$\Omega$}  (m-1-9)
edge node[anchor=east] {$\ast = \gamma$} (m-6-1)
(m-1-9)
edge node[anchor=west] {$\mu$} (m-6-1)
; 
\end{tikzpicture}
\end{equation*}

\end{visual}

Lastly, for i) (sweep around) a visualization is not completely necessary, as we see that the indication of i) is that if $\sigma$ is a simplex of dimension greater than or equal to 2, then there is at most one action of $\sigma$.  In other words, any two actions of $\sigma$ are equal.

\section{Determining the Coefficients for a Few Spaces}
\label{EXAMPLES}

In this Section, we will use the techniques discussed in Section \ref{VIS} to show how one might determine the possible coefficient modules for a given simplicial set.  Before doing so, we consider the following:

\begin{remark}
When determining the possible coefficient modules, we can simply consider and identify the actions among non-degenerate simplices, since for a degenerate $n+1$-simplex $s_j(\sigma) ,$ if $s_j(\sigma)$ has an action $\Lambda_{(i,n)}^{s_j(\sigma)} ,$ then $d_i(s_j(\sigma))=\ast$ which gives that if $i<j$ we have that $s_{j-1}d_i(\sigma) =\ast$ so $d_i(\sigma)=\ast$ and $\sigma$ has action $\Lambda_{(i,n-1)}^{\sigma}$.  Furthermore, by v) we get $\Lambda_{(i,n)}^{s_j(\sigma)}=\Lambda_{(i,n-1)}^{\sigma} .$  Similarly, if $i\geq j+1$ then $\sigma$ has action $\Lambda_{(i-1,n-1)}^{\sigma}$ which is equal to $\Lambda_{(i,n)}^{s_j(\sigma)} $ which gives us that no additional actions come from degenerate simplices.

To see that new identifications do not take place, we notice if $i<r$ then $d_id_r(s_j(\sigma))=\ast$ and $d_{r-1}d_i(s_j(\sigma))=\ast.$  If $i<j$ this gives that $\Lambda_{(i,n-1)}^{d_rs_j(\sigma)}=\Lambda_{(j-1,n-1)}^{d_is_j(\sigma)} .$  However, any face of the degenerate $n+1$-simplex $s_j(\sigma)$ is degenerate via a face of $\sigma$ and it can be checked that the actions identified above would also be identified with part iii) using the faces of $\sigma.$
\end{remark}

We can now consider a few interesting examples:

\begin{example}
The figure below gives an illustration for the minimal simplicial decomposition of the Torus:

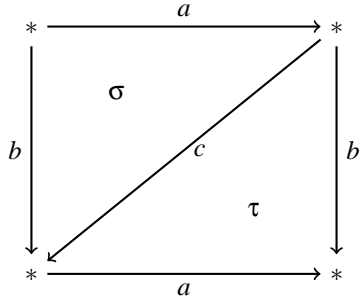
\begin{equation*}
\begin{tikzpicture}[baseline=(current bounding box.center)]
\matrix (m) 
[matrix of math nodes, row sep=1em, column sep=1em, text height=1.5ex, text depth=0.25ex]
{
* & &        & & & & & & * \\
  & & \sigma & & & & & &  \\
  & &        & & & & & &   \\
  & &        & & & & & &  \\ 
  & &        & & & &\tau & &  \\
* & &        & & & & & & *\\
};
\path[->, thick,]
(m-1-1)
edge node[anchor=south] {$a$}  (m-1-9)
edge node[anchor=east] {$b$} (m-6-1)
(m-1-9)
edge node[anchor=west] {$c$} (m-6-1)
(m-1-9)
edge node[anchor=west] {$b$} (m-6-9)
(m-6-1)
edge node[anchor=north]  {$a$} (m-6-9)  
; 
\end{tikzpicture}
\end{equation*}

It is immediately evident from the picture above that there is a forward action of $a$, which must agree with the backward action of $c$, since they are both faces of $\sigma $ which each have a face of $\ast.$  The other implications coming from $\sigma$ are that the forward action of $c$ is equal to the forward action of $b$ and the backward action of $a$ is equal to the backward action of $b.$  Similarly the backward action of $c$ must agree with the forward action of $b$ through $\tau ,$ and we should get that the forward action of $a$ is equal to the forward action of $b$ and the forward action of $c$ is equal to the backward action of $a$.  From the three $1$-simplices $a, b$ and $c$ it can be seen that we would start with 6 actions (one for each direction), but through our identifications, we get that these all must be the same.  Since no other simplices have a face of $\ast$ we get the following Proposition. 

\end{example}

\begin{proposition}
For the minimal simplicial decomposition of the Torus, $HH^{\ast}$ takes only uni-module coefficients.
\end{proposition}

\begin{example}

The figure below is an illustration for the minimal simplicial decomposition of the pinched Torus, where the $1^{st}$ face of $\sigma$ and the $1^{st}$ face of $\tau$ are identified with $\ast$
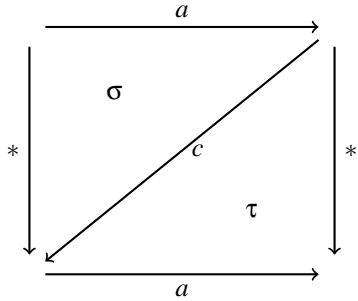
\begin{equation*}
\begin{tikzpicture}[baseline=(current bounding box.center)]
\matrix (m) 
[matrix of math nodes, row sep=1em, column sep=1em, text height=1.5ex, text depth=0.25ex]
{
 & &        & & & & & &  \\
  & & \sigma & & & & & &  \\
  & &        & & & & & &   \\
  & &        & & & & & &  \\ 
  & &        & & & &\tau & &  \\
 & &        & & & & & & \\
};
\path[->, thick,]
(m-1-1)
edge node[anchor=south] {$a$}  (m-1-9)
edge node[anchor=east] {$\ast$} (m-6-1)
(m-1-9)
edge node[anchor=west] {$c$} (m-6-1)
(m-1-9)
edge node[anchor=west] {$\ast$} (m-6-9)
(m-6-1)
edge node[anchor=north]  {$a$} (m-6-9)  
; 
\end{tikzpicture}
\end{equation*}

Here we see that the $1^{st}$ action of $\sigma$, the backward action of $a$ and the forward action of $c$ are all equal, while the $1^{st}$ action of $\tau$, the forward action of $a$ and the backward action of $c$ are all equal.  No other identifications can be made, so we have the following Proposition.
\end{example}

\begin{proposition}
For the minimal simplicial decomposition of the pinched Torus, $HH^{\ast}$ can take coefficients in any bi-module.
\end{proposition}

\begin{example}
\label{classical}
In the case of $S^n$ (when we consider the minimal simplicial decompositions with one $n$-dimensional non-degenerate simplex) we get the following:

For the classical case when $n=1$ we are allowed bi-module coefficients, since there is a $1$-simplex with both a forward and backward action (see below)

\begin{equation*}
\begin{tikzpicture}[baseline=(current bounding box.center)]
\matrix (m) 
[matrix of math nodes, row sep=1em, column sep=1em, text height=1.5ex, text depth=0.25ex]
{
* & & & & & & & & * \\
};
\path[->, thick,]
(m-1-1)
edge node[anchor=south] {$a$}  (m-1-9)
; 
\end{tikzpicture}
\end{equation*} 

When $n$ is larger than $1$ we see that there is exactly one $n$-simplex, which has $\ast$ as every face, so by i), every arising action is identified, leaving us with the following Proposition.

\end{example}

\begin{proposition}
For the minimal simplicial decomposition of $S^n$ with $n>1$, $HH^{\ast}$ can only take coefficients in uni-modules.
\end{proposition}

\noindent{\bf Acknowledgment. \ }
I would like to thank Andrew Salch for introducing me to Hochschild cohomology and for his many hours of discussing the topics presented in this paper and the referee for his or her suggestions on how to improve the paper.  I would like to thank my wife, Kendall for her continued support.  Lastly, I would like to thank my newborn son, Amos for allowing me to get the mathematics in this paper written down before joining the world.









\end{document}